\documentclass[a4paper, 12pt]{article}

\title{Contact three-manifolds with exactly two simple Reeb orbits}
\date{\today}
\author{Dan Cristofaro-Gardiner\thanks{Partially supported by NSF grants DMS-1711976 and DMS-2105471.}, Umberto Hryniewicz\thanks{Partially supported by the DFG SFB/TRR 191 `Symplectic Structures in Geometry, Algebra and Dynamics', Projektnummer 281071066-TRR 191.},\\ Michael Hutchings\thanks{Partially supported by NSF grant DMS-2005437.},
Hui Liu\thanks{Partially supported by NSFC (Nos. 12022111, 11771341) and
the Fundamental Research Funds for the Central Universities (No. 2042021kf1059).}}

\usepackage[utf8]{inputenc}
\usepackage[english]{babel}
\usepackage[dvipsnames]{xcolor}
\usepackage{amssymb}
\usepackage{latexsym}
\usepackage{amsmath}
\usepackage{amsthm}
\usepackage{amscd}
\usepackage{enumerate}
\usepackage{overpic}
\usepackage[mathscr]{euscript}
\usepackage{amsrefs}

\numberwithin{equation}{section}

\newcommand{\floor}[1]{\left\lfloor #1 \right\rfloor}
\newcommand{\ceil}[1]{\left\lceil #1 \right\rceil}

\renewcommand{\frak}{\mathfrak}

\newtheorem{theorem}{Theorem}[section]
\newtheorem{proposition}[theorem]{Proposition}
\newtheorem{corollary}[theorem]{Corollary}
\newtheorem{lemma}[theorem]{Lemma}
\newtheorem{lemma-definition}[theorem]{Lemma-Definition}

\theoremstyle{definition}
\newtheorem{definition}[theorem]{Definition}
\newtheorem{remark}[theorem]{Remark}

\newtheorem{example}[theorem]{Example}

\newtheorem{notation}[theorem]{Notation}

\newcommand{\eqdef}{\;{:=}\;}

\renewcommand{\frak}{\mathfrak}

\newcommand{\C}{{\mathbb C}}
\newcommand{\Q}{{\mathbb Q}}
\newcommand{\R}{{\mathbb R}}

\newcommand{\Z}{{\mathbb Z}}

\newcommand{\op}{\operatorname}

\newcommand{\Ker}{\op{Ker}}

\newcommand{\tensor}{\otimes}

\renewcommand{\epsilon}{\varepsilon}
\newcommand{\mc}[1]{\mathcal{#1}}

\definecolor{michael}{rgb}{0,.6,0}
\definecolor{umberto}{rgb}{0.773,0.294,0.549}
\definecolor{liu}{rgb}{1.0,0.3,0}
\definecolor{dan}{rgb}{.5,.2,1}

\begin{document}

\maketitle
\begin{abstract}
It is known that every contact form on a closed three-manifold has at least two simple Reeb orbits, and a generic contact form has infinitely many.  We show that if there are exactly two simple Reeb orbits, then the contact form is nondegenerate. Combined with a previous result, this implies that the three-manifold is diffeomorphic to the three-sphere or a lens space, and the two simple Reeb orbits are the core circles of a genus one Heegaard splitting. We also obtain further information about the Reeb dynamics and the contact structure. For example the Reeb flow has a disk-like global surface of section and so its dynamics are described by a pseudorotation; the contact structure is universally tight; and in the case of the three-sphere, the contact volume and the periods and rotation numbers of the simple Reeb orbits satisfy the same relations as for an irrational ellipsoid.
\end{abstract}

\tableofcontents

\section{Introduction}

\subsection{Statement of results}
\label{sec:sor}

Let $Y$ be a closed oriented three-manifold.  Recall that a {\em contact form\/} on $Y$ is a one-form $\lambda$ on $Y$ such that $\lambda \wedge d \lambda > 0$.  A contact form $\lambda$ has an associated {\em Reeb vector field} $R$ defined by the equations
\[
d\lambda(R, \cdot) = 0, \quad \lambda(R) = 1.
\]
A {\em Reeb orbit} is a periodic orbit of $R$, i.e.\ a map
\[
\gamma: \mathbb{R}/ T \mathbb{Z} \longrightarrow Y, \quad \gamma'(t) = R(\gamma(t)),
\]
for some $T>0$, modulo reparametrization of the domain by translations. The number $T$ is the period, also called the {\em symplectic action\/}, of $\gamma$. We say that the Reeb orbit $\gamma$ is {\em simple} if the map $\gamma$ is an embedding. Every Reeb orbit is the $k$-fold cover of a simple Reeb orbit for some positive integer $k$.

The three-dimensional case of the Weinstein conjecture, which was proved in full generality by Taubes \cite{taubes-weinstein}, asserts that a contact form on a closed three-manifold has at least one Reeb orbit; see \cite{tw} for a survey. It was further shown in \cite{onetwo} that a contact form on a closed three-manifold has at least two simple Reeb orbits. This lower bound is the best possible without further hypotheses:

\begin{example}
\label{ex:ellipsoid}
Recall that if $Y$ is a compact hypersurface in $\R^4=\C^2$ which is ``star-shaped'' (transverse to the radial vector field), then the standard Liouville form 
\begin{equation}
\label{eqn:Liouville}
\lambda= \frac{1}{2}\sum_{i=1}^2(x_i\,dy_i-y_i\,dx_i)
\end{equation}
restricts to a contact form on $Y$. If $Y$ is the three-dimensional ellipsoid
\[
\partial E(a,b) = \left\{z\in\C^2\;\bigg|\; \frac{\pi|z_1|^2}{a} + \frac{\pi|z_2|^2}{b} = 1\right\},
\]
and if $a/b$ is irrational, then there are exactly two simple Reeb orbits, corresponding to the circles in $Y$ where $z_2=0$ and $z_1=0$, with periods $a$ and $b$ respectively.

One can also take quotients of the above irrational ellipsoids by finite cyclic group actions to obtain contact forms on lens spaces with exactly two simple Reeb orbits.
\end{example}

It is conjectured that in fact, every contact form on a closed connected three-manifold has either two or infinitely many simple Reeb orbits.  This was proved  in~\cite{cdr} for contact forms that are nondegenerate (see the definition below), extending a result of~\cite{twoinf}. It was also shown by Irie \cite{irie} that for a $C^\infty$-generic contact form on a closed three-manifold, there are infinitely many simple Reeb orbits, and moreover their images are dense in the three-manifold.

The goal of this paper is to give detailed information about the ``exceptional'' case of contact forms on a closed three-manifold with exactly two simple Reeb orbits. 

To state the first result, let $\xi=\op{Ker}(\lambda)$ denote the contact structure determined by $\lambda$. This is a rank $2$ vector bundle with a linear symplectic form $d\lambda$. If $\gamma:\R/T\Z\to Y$ is a Reeb orbit, then the derivative of the time $T$ flow of $R$ restricts to a symplectic linear map
\begin{equation}
\label{eqn:Pgamma}
P_\gamma: (\xi_{\gamma(0)},d\lambda) \longrightarrow (\xi_{\gamma(0)},d\lambda),
\end{equation}
which we call the {\em linearized return map\/}.
We say that $\gamma$ is {\em nondegenerate} if $1$ is not an eigenvalue of $P_\gamma$; this condition is invariant under reparametrization of $\gamma$.  We say that the contact form $\lambda$ is nondegenerate if all Reeb orbits (including non-simple ones) are nondegenerate. The set of nondegenerate contact forms is residual in the set of all contact forms with the $C^\infty$-topology. The Reeb orbit $\gamma$ is called {\em hyperbolic\/} if $P_\gamma$ has eigenvalues in $\R\setminus\{\pm1\}$. The Reeb orbit $\gamma$ is called {\em elliptic} if the eigenvalues of $P_\gamma$ are of the form $e^{\pm 2\pi i \phi}$, and {\em irrationally elliptic} if moreover $\phi$ is irrational. If $\gamma$ is irrationally elliptic, then $\gamma$ and all of its covers are nondegenerate, because the linearized return map for the $k$-fold cover of $\gamma$ has eigenvalues $e^{\pm 2\pi i k \phi}$.

Many results about Reeb dynamics and related questions assume some kind of nondegeneracy hypothesis or allow only certain kinds of degeneracies.  One of the main points of the present work is that we can derive our results without making any such assumption.

\begin{theorem}
\label{thm:nondeg}
Let $Y$ be a closed three-manifold, and let $\lambda$ be a contact form on $Y$ with exactly two simple Reeb orbits. Then $\lambda$ is nondegenerate, and moreover both simple Reeb orbits are irrationally elliptic.
\end{theorem}

Theorem~\ref{thm:nondeg} might seem surprising in view of known results about critical points of real-valued functions on finite dimensional manifolds.  For example, on the two-torus, the minimal number of critical points is three, and when there are only three critical points they cannot all be nondegenerate.  We refer the reader to Remark~\ref{rem:pseudo} for related discussion.

As a corollary of Theorem~\ref{thm:nondeg}, we obtain the following topological constraint:

\begin{corollary}
\label{cor:lensspace}
Let $Y$ be a closed three-manifold, and let $\lambda$ be a contact form on $Y$ with exactly two simple Reeb orbits. Then $Y$ is diffeomorphic to a lens space\footnote{Here and below, our convention is that $S^3$ is a lens space, but $S^1 \times S^2$ is not.}.  Moreover, the two simple Reeb orbits are the core circles of a genus one Heegaard splitting of $Y$.
\end{corollary}

\begin{proof}
This was shown in \cite[Thm.\ 1.3 and \S4.8]{wh} under the additional hypothesis that $\lambda$ is nondegenerate. By Theorem~\ref{thm:nondeg}, this nondegeneracy holds automatically.
\end{proof}

\begin{remark}
A special case of Theorem~\ref{thm:nondeg}, where $Y$ is a compact convex hypersurface in $\R^4$ with the restriction of the standard Liouville form \eqref{eqn:Liouville}, was previously shown in \cite[Thm.\ 1.4]{WHL}.
\end{remark}

We also obtain additional dynamical information. To state the result, recall that the {\em contact volume\/} of $(Y,\lambda)$ is defined by
\[
\op{vol}(Y,\lambda) := \int_Y \lambda \wedge d \lambda.
\] 

\begin{theorem}
\label{thm:dynamics}
Let $Y$ be a lens space and let $\lambda$ be a contact form on $Y$ with exactly two simple Reeb orbits, $\gamma_1$ and $\gamma_2$. 
Then:
\begin{description}
\item{(a)}  Let $p = | \pi_1(Y)| < \infty$, let $T_i\in\R$ denote the period of $\gamma_i$, and let $\phi_i\in\R$ denote the ``Seifert rotation number'' of $\gamma_i$, see Definition~\ref{def:srn}. Then
\[
\op{vol}(Y,\lambda) = pT_1T_2 = T^2_1/\phi_1 = T^2_2/\phi_2.
\]
\item{(b)} $\lambda$ is dynamically convex, and the contact structure $\xi=\Ker(\lambda)$ is universally tight\footnote{Recall that a contact form on a three-manifold $Y$ with $c_1(\xi)|_{\pi_2(Y)} = 0$ is called {\em dynamically convex} if CZ$(\gamma) \ge 3$ for every contractible Reeb orbit $\gamma$, where CZ denotes the Conley-Zehnder index (see \S\ref{sec:defnech}) computed with respect to a trivialization which extends over a disc bounded by $\gamma$. A contact structure on $Y$ is {\em universally tight\/} if its pullback to the universal cover of $Y$ is tight.}.
\end{description}
\end{theorem}

\begin{example}
For the ellipsoid in Example~\ref{ex:ellipsoid}, we have $T_1=a$, $T_2=b$, $\phi_1=a/b$, $\phi_2=b/a$, $p=1$, and $\op{vol}=ab$. Thus
Theorem~\ref{thm:dynamics}(a) implies that if $Y=S^3$, then the periods $T_i$, the rotation numbers $\phi_i$, and the contact volume satisfy the same relations as for an ellipsoid. For $Y=S^3$, under the additional assumptions that $\lambda$ is nondegenerate and $\xi$ is the standard contact structure, it was previously shown in \cite{BCE,gurel}, that ``action-index relations'' hold, implying that the periods $T_i$ and rotation numbers $\phi_i$ satisfy the same relations as for an ellipsoid. The equation $\op{vol}=T_1T_2$ that we prove in this case  answers~\cite[Question~2]{action_linking}.  
\end{example}

\begin{remark}
There exist contact forms on $S^3$ with exactly two simple Reeb orbits which are not strictly contactomorphic to ellipsoids. One way to see this is to start from Katok's construction~\cite{katok} of Finsler metrics on $S^2$ with exactly two closed geodesics, such that the Liouville measure on the unit tangent bundle is ergodic for the geodesic flow. Such a geodesic flow can then be lifted to a Reeb flow on the standard contact $3$-sphere with the same properties. Another way to see this is by Albers-Geiges-Zehmisch \cite{agz}, who showed that the pseudorotations from \cite{fayad_katok} can be realized as the return map on a disk-like global surface of section for a Reeb flow on the standard contact $3$-sphere with precisely two periodic orbits; see~\S\ref{sec:pseudo} below. On the other hand, Helmut Hofer has suggested to the authors in private correspondence that perhaps imposing the additional condition that the rotation numbers of the two Reeb orbits are Diophantine forces the contact form to be strictly contactomorphic to an ellipsoid; cf \cite[Question 6]{hoferconjecture}. 
\end{remark}

\begin{remark}
As shown in \cite[Prop. 5.1]{honda} (see \cite[p. 17]{corn} for more explanation), each lens space has either one or two universally tight contact structures up to isotopy, and when there are two they are contactomorphic (and one is obtained from the other by reversing its orientation). Consequently, in Theorem~\ref{thm:dynamics}(b), the contact structure is contactomorphic to a ``standard'' contact structure on the lens space obtained as in Example~\ref{ex:ellipsoid}. In particular, universally tight contact structures on lens spaces are precisely the ones that admit contact forms with exactly two simple Reeb orbits. Some other results obtaining information about contact structures from Reeb dynamics can be found in \cite{hwzduke, hwzcpam, HLS, etnyreghrist}.
\end{remark}

\begin{remark}
\label{rmk:knottype}
We also obtain information about the knot types of the simple Reeb orbits $\gamma_1$ and $\gamma_2$. It follows from the Heegaard splitting in Corollary~\ref{cor:lensspace} that these are $p$-unknotted. We further show in \S\ref{sec:fc} that their self-linking number is $-1$ when $p=1$;  similar arguments show that for general $p$, their rational self-linking number, as defined in \cite{be}, equals $-1/p$.
\end{remark}

\subsection{Outline of the proofs}

We now briefly describe the proofs of Theorems~\ref{thm:nondeg} and \ref{thm:dynamics}.

A key ingredient in these proofs, as well as in the related papers \cite{onetwo,twoinf}, is the ``Volume Property'' in embedded contact homology, which was proved in \cite{asymp}. The embedded contact homology (ECH) of $(Y,\lambda)$ is the homology of a chain complex which is built out of Reeb orbits, and whose differential counts (mostly) embedded pseudoholomorphic curves in $\R\times Y$; see the lecture notes \cite{bn} and the review in \S\ref{sec:preliminaries}. The version of the Volume Property that we will use here asserts that if $Y$ is a closed connected $3$-manifold with a contact form $\lambda$, then
\[
\lim_{k\to\infty} \frac{c_{\sigma_k}(Y,\lambda)^2}{k} = 2\op{vol}(Y,\lambda).
\]
Here $\{\sigma_k\}$ is a ``$U$-sequence'' in ECH, and $c_{\sigma_k}$ is a ``spectral invariant'' associated to $\sigma_k$, which is the total symplectic action of a certain finite set of Reeb orbits determined by $\sigma_k$; these notions are reviewed in \S\ref{sec:preliminaries}. 

The outline of the proof of Theorem~\ref{thm:nondeg} is as follows. Let $\gamma_1$ and $\gamma_2$ denote the two simple Reeb orbits, and let $T_1$ and $T_2$ denote their periods. Simple applications of the Volume Property from \cite{onetwo,twoinf} (just using the $k^{1/2}$ growth rate of the spectral invariants and not the exact relation with contact volume) show that the homology classes $[\gamma_i]\in H_1(Y)$ are torsion, and the ratio $T_1/T_2$ is irrational. A more precise use of the Volume Property then gives the relations
\begin{equation}
\label{eqn:cv}
\phi_i = \frac{T_i^2}{\op{vol}(Y,\lambda)} \quad\quad\quad
\ell(\gamma_1,\gamma_2) = \frac{T_1T_2}{\op{vol}(Y,\lambda)}
\end{equation}
where $\phi_i \in \R$ is the Seifert rotation number that appears in Theorem~\ref{thm:dynamics}(a), while $\ell(\gamma_1,\gamma_2)\in\Q$ is the linking number of $\gamma_1$ and $\gamma_2$, see Definition~\ref{def:rln}. The proof of \eqref{eqn:cv} also depends on a new estimate for the behavior of the ECH index (the grading on the ECH chain complex) under perturbations of possibly degenerate contact forms, which is proved in \S\ref{sec:estimate}.

The equations \eqref{eqn:cv} imply the relations
\begin{equation}
\label{eqn:actionindex}
\phi_1 = \ell(\gamma_1,\gamma_2)\frac{T_1}{T_2}, \quad\quad\quad \phi_2 = \ell(\gamma_1,\gamma_2)\frac{T_2}{T_1}.
\end{equation}
Since $\ell(\gamma_1,\gamma_2)$ is rational and $T_1/T_2$ is irrational, it follows that $\phi_1$ and $\phi_2$ are irrational. The latter fact implies that $\gamma_1$ and $\gamma_2$ are irrationally elliptic; see \S\ref{sec:cpn}. This completes the proof of Theorem~\ref{thm:nondeg}.

The Heegaard decomposition in Corollary~\ref{cor:lensspace} implies that $\ell(\gamma_1,\gamma_2)=1/p$, and combined with \eqref{eqn:actionindex} this proves Theorem~\ref{thm:dynamics}(a). The proof of Theorem~\ref{thm:dynamics}(b) uses additional calculations in \S\ref{sec:fc} to deduce dynamical convexity and universal tightness from information about the numbers $\phi_i$.

\subsection{Pseudorotations}
\label{sec:pseudo}

The contact forms studied here are analogous to ``pseudorotations'', defined in various ways as maps in some class with the minimum number of periodic orbits. For example, in \cite{gg} a {\em Hamiltonian pseudorotation\/} of $\C P^n$ is defined to be a Hamiltonian symplectomorphism of $\C P^n$ with $n+1$ fixed points and no other periodic points (see e.g.\ \cite{cgg,ls,shelukhin} for generalizations to other symplectic manifolds).  More classically, we consider here pseudorotations of the open or closed disk defined as area-preserving homeomorphisms with one fixed point and no other periodic points;  see e.g.\ \cite{barneyannals,fayad_katok}. 

In fact, there is a direct connection between the contact forms considered in this paper and pseudorotations of the closed disk.  Recall that given a closed three-manifold $Y$ with a contact form $\lambda$, a {\em disk-like global surface of section} for the Reeb flow is an immersed disk, with boundary on a Reeb orbit, embedded and transverse to the Reeb flow in the interior, such that the Reeb flow starting at any point in $Y$ hits the disk both forwards and backwards in time.

\begin{corollary}
\label{cor:disc}
Let $Y$ be a closed three-manifold, and let $\lambda$ be a contact form on $Y$ with exactly two simple Reeb orbits. Then both orbits bound disk-like global surfaces of section whose associated return maps define smooth pseudorotations of the open disk. 
\end{corollary} 

\begin{proof}
By Theorem~\ref{thm:nondeg}, Corollary~\ref{cor:lensspace} and Theorem~\ref{thm:dynamics}, $Y$ is a lens space and $\lambda$ is nondegenerate and dynamically convex.  As explained in Remark~\ref{rmk:knottype}, both orbits are $p$-unknotted, with self-linking number $-1/p$.  Hence, the result follows from \cite[Thm. 1.12]{HLS}.      
\end{proof}

\begin{remark}
Conversely, as mentioned above, at least some pseudorotations of the closed disk can be ``suspended" to contact forms on $S^3$ with exactly two simple Reeb orbits  \cite{agz}. 
\end{remark}

\begin{remark}
\label{rem:pseudo}
It is shown in \cite{kerman}, see also \cite{franks}, that for a Hamiltonian pseudorotation of $\C P^1$, each fixed point is strongly nondegenerate, meaning that the linearized return map and its higher powers are nondegenerate, and moreover the fixed points are irrationally elliptic, similarly to Theorem~\ref{thm:nondeg}.  It is an open question whether every pseudorotation of $\C P^n$ for $n > 1$ is strongly nondegenerate, and one can ask analogous questions for pseudorotations of more general symplectic manifolds. 
\end{remark}

\begin{remark}
For pseudorotations of $D^2$ in many cases, identities related to Theorem~\ref{thm:dynamics}(a) were proved in \cite{abror,joli}.
\end{remark}

\begin{remark}
One can arrange in the statement of Corollary~\ref{cor:disc} that the first return maps on the obtained disk-like global surfaces of section extend smoothly to the boundary and preserve a smooth $2$-form that defines an area-form in the interior. Moreover, one can conjugate such a return map by a homeomorphism to obtain a pseudorotation of the closed disk which is smooth in the interior. 
\end{remark}

\paragraph{Acknowledgments.} Key discussions for this project took place when DCG and HL were visiting Peking University for the ``Workshop on gauge theory and Floer homology'' in December 2019. We thank Peking University for their hospitality during this visit. This research was completed while DCG was on a von Neumann fellowship at the Insititute for Advanced Study. DCG thanks the Institute for their wonderful support. We thank Viktor Ginzburg and Helmut Hofer for helpful comments. We thank the referee for careful reading of the paper.


\section{Preliminaries}
\label{sec:preliminaries}

In this section we review the material about embedded contact homology that is needed for the proofs of Theorems~\ref{thm:nondeg} and \ref{thm:dynamics}. We include a new, slight extension of the definition of the ECH index to degenerate contact forms.

Throughout this section fix a closed oriented three-manifold $Y$ and a contact form $\lambda$ on $Y$, and let $\xi=\Ker(\lambda)$ denote the associated contact structure.

\subsection{Topological preliminaries}
\label{sec:topological}

We now recall some topological notions we will need, following the treatment in \cite{ir}. These were originally introduced in a slightly different context in \cite{pfh2}.

\begin{definition}
An {\em orbit set\/} is a finite set of pairs $\alpha=\{(\alpha_i,m_i)\}$ where the $\alpha_i$ are distinct simple Reeb orbits, and the $m_i$ are positive integers. We define the homology class of the orbit set $\alpha$ by
\[
[\alpha] = \sum_i m_i[\alpha_i]\in H_1(Y).
\]
\end{definition}

\begin{definition}
If $\alpha=\{(\alpha_i,m_i)\}$ and $\beta=\{(\beta_j,n_j)\}$ are orbit sets with $[\alpha]=[\beta]$, define $H_2(Y,\alpha,\beta)$ to be the set of $2$-chains $Z$ in $Y$ with $\partial Z = \sum_i m_i\alpha_i - \sum_j n_j\beta_j$, modulo boundaries of $3$-chains. The set $H_2(Y,\alpha,\beta)$ is an affine space over $H_2(Y)$. 
\end{definition}

Given orbit sets $\alpha$ and $\beta$ as above, let $Z\in H_2(Y,\alpha,\beta)$, and let $\tau$ be a homotopy class of symplectic trivialization of the contact structure $\xi$ over the Reeb orbits $\alpha_i$ and $\beta_j$.

\begin{definition}
(cf.\ \cite[\S2.5]{ir})
Define the {\em relative first Chern class\/}
\[
c_\tau(\alpha,\beta,Z) \in \Z
\]
as follows. Let $S$ be a compact oriented surface with boundary and let $f:S\to Y$ be a smooth map representing the class $Z$. Let $\psi$ be a section of $f^*\xi$ which on each boundary component is nonvanishing and constant with respect to $\tau$. Define $c_\tau(\alpha,\beta,Z)$ to be the algebraic count of zeroes of $\psi$.
\end{definition}

\begin{definition}
\label{def:admrep}
\cite[\S2.7]{ir}
An {\em admissible representative\/} of $Z\in H_2(Y,\alpha,\beta)$ is a smooth map $f:S\to[-1,1]\times Y$ where $S$ is a compact oriented surface with boundary; the restriction of $f$ to $\partial S$ consists of positively oriented covers of $\{1\}\times\alpha_i$ with total multiplicity $m_i$ and negatively oriented covers of $\{-1\}\times\beta_j$ with total multiplicity $n_j$; the composition of $f$ with the projection $[-1,1]\times Y\to Y$ represents the class $Z$; the restriction of $f$ to the interior of $S$ is an embedding; and $f$ is transverse to $\{-1,1\}\times Y$.
\end{definition}

\begin{definition}
\cite[\S2.7]{ir}
If $Z,Z'\in H_2(Y,\alpha,\beta)$, define the {\em relative intersection pairing\/}
\[
Q_\tau(Z,Z')\in\Z
\]
as follows. Let $S,S'$ be admissible representatives of $Z$ and $Z'$ respectively whose interiors are transverse and do not intersect near the boundary. Define
\begin{equation}
\label{eqn:Q}
Q_\tau(Z,Z') = \#(\op{int}(S)\cap\op{int}(S')) - \sum_i\ell_\tau(\zeta_i^+,{\zeta_i^+}') + \sum_j\ell_\tau(\zeta_j^-,{\zeta_j^-}').
\end{equation}
Here `$\#$' denotes the signed count of intersections, while the remaining terms are linking numbers defined as follows. For $\epsilon>0$ small, the intersection of $S$ with $\{1-\epsilon\}\times Y$ consists of the union over $i$ of a braid $\zeta_i^+$ in a neighborhood of $\alpha_i$ (see \S\ref{sec:rls}), while the intersection of $S$ with $\{-1+\epsilon\}\times Y$ consists of the union over $j$ of a braid $\zeta_j^-$ in a neighborhood of $\beta_j$. Likewise, $S'$ determines braids ${\zeta_i^+}'$ and ${\zeta_j^-}'$. The notation $\ell_\tau$ indicates the linking number in a neighborhood of $\alpha_i$ or $\beta_j$ computed using the trivialization $\tau$; see \cite[\S2.6]{ir} for details and sign conventions.

When $Z=Z'$, we write\footnote{An alternate, equivalent definition of $Q_\tau(\alpha,\beta,Z)$ is given in \cite[\S3.3]{bn}, which does not include the linking number terms in \eqref{eqn:Q}. There the admissible representatives $S$ and $S'$ are required to satisfy additional conditions which force these linking number terms to be zero.}
\[
Q_\tau(\alpha,\beta,Z) = Q_\tau(Z,Z).
\]
\end{definition}

As explained in \cite{ir}, the relative first Chern class $c_\tau(\alpha,\beta,Z)$ and the relative self-intersection number $Q_\tau(\alpha,\beta,Z)$ depend only on $\alpha$, $\beta$, $Z$, and $\tau$. Moreover, if we change $Z$ by adding $A\in H_2(Y)$ then they behave as follows:
\begin{align}
\label{eqn:camb}
c_\tau(\alpha,\beta,Z+A) - c_\tau(\alpha,\beta,Z) &= \langle c_1(\xi),A\rangle,\\
\label{eqn:Qamb}
Q_\tau(\alpha,\beta,Z+A) - Q_\tau(\alpha,\beta,Z) &= 2[\alpha]\cdot A.
\end{align}

\begin{remark}
\label{rem:additive}
If $\gamma$ is a third orbit set, if $\tau$ is a trivialization of $\xi$ over the Reeb orbits in $\alpha$, $\beta$, and $\gamma$, and if $W\in H_2(Y,\beta,\gamma)$, then we have the additivity properties
\[
\begin{split}
c_\tau(\alpha,\beta,Z) + c_\tau(\beta,\gamma,W) &= c_\tau(\alpha,\gamma,Z+W),\\
Q_\tau(\alpha,\beta,Z) + Q_\tau(\beta,\gamma,W) &= Q_\tau(\alpha,\gamma,Z+W).
\end{split}
\]
Note also that the definition of $c_\tau$ makes sense more generally if the $\alpha_i$ and $\beta_j$ are transverse knots. Likewise the definition of $Q_\tau$ makes sense if the $\alpha_i$ and $\beta_j$ are knots and $\tau$ is an oriented trivialization of their normal bundles. 
\end{remark}

\subsection{The ECH index}
\label{sec:defnech}

Let $\gamma:\R/T\Z\to Y$ be a Reeb orbit and let $\tau$ be a symplectic trivialization of $\gamma^*\xi$.
The derivative of the time $t$ Reeb flow from $\xi_{\gamma(0)}$ to $\xi_{\gamma(t)}$, with respect to $\tau$, is a $2\times 2$ symplectic matrix $\Phi(t)$. The family of symplectic matrices $\{\Phi(t)\}_{t\in[0,T]}$ induces a family of diffeomorphisms of $S^1$ in the universal cover of $\op{Diff}(S^1)$, which has a dynamical rotation number, which we denote by $\theta_\tau(\gamma)\in\R$. We call this real number the {\em rotation number\/} of $\gamma$ with respect to $\tau$ and denote it by $\theta_\tau(\gamma)\in\R$; it depends only on $\gamma$ and the homotopy class of $\tau$. When $\theta_\tau(\gamma)\notin\frac{1}{2}\Z$, the eigenvalues of the linearized return map \eqref{eqn:Pgamma} are $e^{\pm 2\pi i\theta_\tau(\gamma)}$.

\begin{definition}
Define the {\bf Conley-Zehnder index\/}
\begin{equation}
\label{eqn:CZdef}
\op{CZ}_\tau(\gamma) = \floor{\theta_\tau(\gamma)} + \ceil{\theta_\tau(\gamma)}\in\Z.
\end{equation}
\end{definition}

\begin{remark}
The above definition agrees with the usual Conley-Zehnder index when $\gamma$ is nondegenerate. When $\gamma$ is degenerate, it is common to give a different definition of the Conley-Zehnder index, as the minimum of the Conley-Zehnder indices of nondegenerate perturbations of $\gamma$, and this will sometimes differ from our definition by $1$. For the purposes of this paper, especially to obtain an estimate as in Proposition~\ref{prop:perturb} below (possibly with a different constant), it does not matter which of these definitions of the Conley-Zehnder index we use for degenerate Reeb orbits.
\end{remark}

\begin{notation}
If $\alpha=\{(\alpha_i,m_i)\}$ is an orbit set and if $\tau$ is a trivialization of $\xi$ over all of the Reeb orbits $\alpha_i$, define
\begin{equation}
\label{eqn:CZI}
\op{CZ}_\tau^I(\alpha) = \sum_i\sum_{k=1}^{m_i}\op{CZ}_\tau(\alpha_i^k).
\end{equation}
Here $\gamma^k$ denotes the $k^{th}$ iterate of $\gamma$.
\end{notation}

\begin{definition}
Let $\alpha$ and $\beta$ be orbit sets with $[\alpha]=[\beta]\in H_1(Y)$, and let $Z\in H_2(Y,\alpha,\beta)$. Define the {\em ECH index\/}
\begin{equation}
\label{eqn:I}
I(\alpha,\beta,Z) = c_\tau(\alpha,\beta,Z) + Q_\tau(\alpha,\beta,Z) + \op{CZ}_\tau^I(\alpha) - \op{CZ}_\tau^I(\beta) \in \Z.
\end{equation}
\end{definition}

The above agrees with the usual definition of the ECH index, see e.g.\ \cite[\S3.4]{bn}, when the contact form is nondegenerate. It is explained for example in \cite[\S2.8]{ir} why $I(\alpha,\beta,Z)$ depends only on $\alpha$, $\beta$, and $Z$, and not on $\tau$. Moreover, it follows from \eqref{eqn:camb} and \eqref{eqn:Qamb} that if we change $Z$ by adding $A\in H_2(Y)$, then
\begin{equation}
\label{eqn:indexambig}
I(\alpha,\beta,Z+A) - I(\alpha,\beta,Z) = \langle c_1(\xi)+2\op{PD}(\Gamma),A\rangle,
\end{equation}
where $\Gamma=[\alpha]=[\beta]\in H_1(Y)$ and $\op{PD}$ denotes the Poincar\'e dual. By Remark~\ref{rem:additive}, we have
\begin{equation}
\label{eqn:Iadditive}
I(\alpha,\beta,Z) + I(\beta,\gamma,W) = I(\alpha,\gamma,Z+W).
\end{equation}

\subsection{Embedded contact homology}

In this subsection assume that the contact form $\lambda$ is nondegenerate. Let $\Gamma\in H_1(Y)$. We now review how to define the embedded contact homology $ECH_*(Y,\xi,\Gamma)$. More details may be found in \cite{bn}.

\begin{definition}
An {\em ECH generator\/} is an orbit set $\alpha=\{(\alpha_i,m_i)\}$ such that $m_i=1$ whenever $\alpha_i$ is hyperbolic.
\end{definition}

\begin{definition}
Define $ECC_*(Y,\lambda,\Gamma)$ to be the vector space\footnote{It is also possible to use $\Z$ coefficients, as explained in \cite[\S9]{obg2}, but this has not been necessary for the applications of ECH so far.} over $\Z/2$ generated by ECH generators $\alpha$ with $[\alpha]=\Gamma$. This vector space has a relative $\Z/d$ grading, where $d$ denotes the divisibility of $c_1(\xi)+2\op{PD}(\Gamma)\in H^2(Y;\Z)$; if $\alpha$ and $\beta$ are two generators, then their grading difference is $I(\alpha,\beta,Z) \mod d$ for any $Z\in H_2(Y,\alpha,\beta)$.  This makes sense by \eqref{eqn:indexambig} and \eqref{eqn:Iadditive}.
\end{definition}

\begin{remark}
\label{rem:absolute}
In the special case where $c_1(\xi)\in H^2(Y;\Z)$ is torsion and $\Gamma=0$, the chain complex $ECC_*(Y,\lambda,0)$ has a canonical absolute $\Z$-grading defined by
\[
I(\alpha) = I(\alpha,\emptyset,Z)\in \Z
\]
for any $Z\in H_2(Y,\alpha,\emptyset)$. This is well-defined by \eqref{eqn:indexambig}.
\end{remark}

\begin{definition}
An almost complex structure $J$ on $\R\times Y$ is {\em $\lambda$-compatible\/} if $J\partial_s=R$, where $s$ denotes the $\R$ coordinate; $J$ is invariant under the $\R$ action on $\R\times Y$ by translation of $s$; and $J(\xi)=\xi$, rotating positively with respect to $d\lambda$.
\end{definition}

If $J$ is a generic $\lambda$-compatible almost complex structure, one defines a differential
\[
\partial_J:ECC_*(Y,\lambda,\Gamma) \longrightarrow ECC_{*-1}(Y,\lambda,\Gamma)
\]
whose coefficient from $\alpha$ to $\beta$ is a count of ``$J$-holomorphic currents'' that represent classes $Z\in H_2(Y,\alpha,\beta)$ with ECH index $I(\alpha,\beta,Z)=1$. See \cite[\S3]{bn} for details. It is shown in \cite{obg1} that $\partial_J^2=0$. The {\em embedded contact homology\/} $ECH_*(Y,\lambda,\Gamma,J)$ is defined to be the homology of the chain complex $(ECC_*(Y,\lambda,\Gamma),\partial_J)$. A theorem of Taubes \cite{taubes-ech}, tensored with $\Z/2$, asserts that there is a canonical isomorphism
\begin{equation}
\label{eqn:taubes}
ECH_*(Y,\lambda,\Gamma,J) = \widehat{HM}^{-*}(Y,\frak{s}_\xi + \op{PD}(\Gamma))\tensor\Z/2,
\end{equation}
where the right hand side is a version of Seiberg-Witten Floer cohomology as defined by Kronheimer-Mrowka \cite{km}, and $\frak{s}_\xi$ is a spin-c structure on $Y$ determined by $\xi$. In particular, ECH depends only\footnote{In a sense, ECH does not depend on the contact structure either; see \cite[Rmk.\ 1.7]{bn} for explanation.} on the triple $(Y,\xi,\Gamma)$, and so we can denote it by $ECH_*(Y,\xi,\Gamma)$.

When $Y$ is connected, there is also a well-defined ``$U$-map''
\begin{equation}
\label{eqn:Umap}
U: ECH_*(Y,\xi,\Gamma) \longrightarrow ECH_{*-2}(Y,\xi,\Gamma).
\end{equation}
This is induced by a chain map
\[
U_{J,z}: (ECC_*(Y,\lambda,\Gamma),\partial_J) \longrightarrow (ECC_{*-2}(Y,\xi,\Gamma),\partial_J)
\]
which counts $J$-holomorphic currents with ECH index $2$ passing through a generic base point $z\in \R\times Y$. The assumption that $Y$ is connected implies that the induced map on homology does not depend on the choice of base point $z$; see \cite[\S2.5]{wh} for details. Taubes showed in \cite[Thm.\ 1.1]{taubes-U} that under the isomorphism \eqref{eqn:taubes}, the map on homology induced by $U_{J,z}$ agrees with a corresponding map on Seiberg-Witten Floer cohomology. We thus obtain a well-defined $U$-map \eqref{eqn:Umap}.

\begin{definition}
A {\em $U$-sequence} for $\Gamma$ is a sequence $\{\sigma_k\}_{k \ge 1}$ where each $\sigma_k$ is a nonzero homogeneous class in $ECH_*(Y,\xi,\Gamma)$, and $U\sigma_{k+1}=\sigma_k$ for each $k\ge 1$.
\end{definition}

We will need the following nontriviality result for the $U$-map, which is proved by combining Taubes's isomorphism \eqref{eqn:taubes} with results from Kronheimer-Mrowka \cite{km}:

\begin{proposition}
\label{prop:Useq}
\cite[Prop.\ 2.3]{twoinf}
If $c_1(\xi)+2\op{PD}(\Gamma)\in H^2(Y;\Z)$ is torsion, then a $U$-sequence for $\Gamma$ exists.
\end{proposition}

\subsection{Spectral invariants}

If $\alpha=\{(\alpha_i,m_i)\}$ is an orbit set, define its {\em symplectic action\/} by
\[
\mc{A}(\alpha) = \sum_i m_i\int_{\alpha_i}\lambda.
\]
Note here that $\int_{\alpha_i}\lambda$ agrees with the period of $\alpha_i$, because $\lambda(R)=1$.

Assume now that $\lambda$ is nondegenerate. For $L\in\R$, define $ECC^L_*(Y,\lambda,\Gamma)$ to be the subspace of $ECC_*(Y,\lambda,\Gamma)$ spanned by ECH generators $\alpha$ with symplectic action $\mc{A}(\alpha) < L$. It follows from the definition of ``$\lambda$-compatible almost complex structure'' that $\partial_J$ maps $ECC^L$ to itself; see \cite[\S1.4]{bn}. We define the {\em filtered ECH\/}  to be the homology of this subcomplex, which we denote by $ECH^L_*(Y,\lambda,\Gamma)$. The inclusion of chain complexes induces a map
\[
\imath_L : ECH^L_*(Y,\lambda,\Gamma) \longrightarrow ECH_*(Y,\xi,\Gamma).
\]
It is shown in \cite[Thm.\ 1.3]{cc2} that the filtered homology $ECH_*^L(Y,\lambda,\Gamma)$ and the map $\imath_L$ do not depend on the choice of $J$. However, unlike the usual ECH, filtered ECH does depend on the contact form $\lambda$ and not just on the contact structure $\xi$.

\begin{definition}
If $0\neq \sigma\in ECH_*(Y,\xi,\Gamma)$, define the {\em spectral invariant\/}
\[
c_\sigma(Y,\lambda) = \inf\{L\mid \sigma\in\op{Im}(\imath_L)\} \in \R.
\]
\end{definition}

An equivalent definition is that $c_\sigma(Y,\lambda)$ is the minimum $L$ such that the class $\sigma$ can be represented by a cycle in the chain complex $(ECC_*(Y,\lambda,\Gamma),\partial_J)$ which is a sum of ECH generators each having symplectic action $\le L$. In particular, by definition $c_\sigma(Y,\lambda) = \mc{A}(\alpha)$ for some ECH generator $\alpha$ with $[\alpha]=\Gamma$.

We can change the contact form $\lambda$, without changing the contact structure $\xi$, by multiplying $\lambda$ by a smooth function $f:Y\to\R^{>0}$. As explained in \cite[\S2.5]{onetwo}, 
it turns out that even when $\lambda$ is degenerate, one can still define $c_\sigma(Y,\lambda)$ as a limit of spectral invariants $c_\sigma(Y,f_n\lambda)$ where $f_n\lambda$ is nondegenerate and $f_n\to 1$ in $C^0$.

These spectral invariants have the following important properties:

\begin{proposition}
Let $Y$ be a closed connected three-manifold, and let $\lambda$ be a (possibly degenerate) contact form on $Y$. Then:
\label{prop:spectral}
\begin{description}
\item{(a)} If $0\neq \sigma\in ECH_*(Y,\xi,\Gamma)$, then
\[
c_\sigma(Y,\lambda)=\mc{A}(\alpha)
\]
for some orbit set $\alpha$ with $[\alpha]=\Gamma$.
\item{(b)}  If $\sigma\in ECH_*(Y,\xi,\Gamma)$ and $U\sigma\neq 0$, then
\begin{equation}
\label{eqn:Uda}
c_{U\sigma}(Y,\lambda) \le c_\sigma(Y,\lambda).
\end{equation}
If there are only finitely many simple Reeb orbits, then the inequality \eqref{eqn:Uda} is strict.
\item{(c)} (``Volume Property'')  If $c_1(\xi)+2\op{PD}(\Gamma)\in H^2(Y;\Z)$ is torsion, and if $\{\sigma_k\}_{k\ge 1}$ is a $U$-sequence for $\Gamma$, then
\[
\lim_{k\to\infty} \frac{c_{\sigma_k}(Y,\lambda)^2}{k} = 2\op{vol}(Y,\lambda).
\]
\end{description}
\end{proposition}

\begin{proof}
As noted above, part (a) holds by definition when $\lambda$ is nondegenerate, and in the degenerate case it follows from a compactness argument for Reeb orbits; cf.\ \cite[Lem.\ 3.1(a)]{onetwo}.

If $\lambda$ is nondegenerate, then since the chain map $U_{J,z}$ counts $J$-holomorphic curves, it decreases symplectic action like the differential, so strict inequality in \eqref{eqn:Uda} holds. The not necessarily strict inequality \eqref{eqn:Uda} in the degenerate case follows by a limiting argument. The fact that the inequality \eqref{eqn:Uda} is strict for degenerate contact forms with only finitely many simple Reeb orbits\footnote{The equality $c_{U\sigma}(Y,\lambda) = c_\sigma(Y,\lambda)$ is possible for degenerate contact forms with infinitely many simple Reeb orbits. This happens for example for some classes $\sigma$ when $Y$ is an ellipsoid $\partial E(a,b)$ with $a/b$ rational.} is proved by a more subtle compactness argument for holomorphic curves in \cite[Lem.\ 3.1(b)]{onetwo}. 

Part (c), the most nontrivial part, is a special case of \cite[Thm.\ 1.3]{asymp}.
\end{proof}


\section{The ECH index and perturbations}
\label{sec:estimate}

The goal of this section is to prove Proposition~\ref{prop:perturb} below, which gives an upper bound on how much the ECH index can change when one perturbs the contact form. This is an important ingredient in the proof of Theorems~\ref{thm:nondeg} and \ref{thm:dynamics}.

To state the proposition, let $\lambda$ be a contact form on a closed three-manifold $Y$, and let $\lambda_n=f_n\lambda$ be a sequence of contact forms with $f_n\to 1$ in $C^2$. In the case of interest, $\lambda$ will be degenerate, while each of the contact forms $\lambda_n$ will be nondegenerate.

Fix an orbit set $\alpha=\{(\alpha_i,m_i)\}$ for $\lambda$, and let $N$ be a disjoint union of tubular neighborhoods $N_i$ of the simple Reeb orbits $\alpha_i$. Consider a sequence of orbit sets $\alpha(n)$ for $\lambda_n$ that converges to $\alpha$ as currents. In particular this implies that if $n$ is sufficiently large, and if we write $\alpha'=\alpha(n)$, then $\alpha'$ is contained in $N$, and its intersection with $N_i$ is homologous in $N_i$ to $m_i\alpha_i$. There is then a unique $W_\alpha\in H_2(Y,\alpha',\alpha)$ that is contained in $N$.

Likewise fix an orbit set $\beta=\{(\beta_j,n_j)\}$ for $\lambda$ along with disjoint tubular neighborhoods of the simple Reeb orbits $\beta_j$, and consider a sequence of orbit sets $\beta(n)$ for $\lambda_n$ that converges to $\beta$ as currents. Then for $k$ sufficiently large, writing $\beta'=\beta(n)$, we obtain a distinguished $W_\beta\in H_2(Y,\beta',\beta)$.

For fixed large $n$ there is now a bijection
\[
H_2(Y,\alpha,\beta)\simeq H_2(Y,\alpha',\beta')
\]
sending $Z\in H_2(Y,\alpha,\beta)$ to
\[
Z' = Z + W_\alpha - W_\beta \in H_2(Y,\alpha',\beta').
\]

\begin{proposition}
\label{prop:perturb}
With the notation as above, for fixed orbit sets $\alpha$ and $\beta$, if $n$ is sufficiently large, then
\[
\big| I(\alpha,\beta,Z) - I(\alpha',\beta',Z') \big| \le 2 \left(\sum_im_i + \sum_jn_j\right).
\]
Here $I(\alpha,\beta,Z)$ denotes the ECH index for $\lambda$, and $I(\alpha',\beta',Z')$ denotes the ECH index for $\lambda_n$.
\end{proposition}

\subsection{Reduction to a local statement}
\label{sec:rls}

We now reduce Proposition~\ref{prop:perturb} to a local statement, Proposition~\ref{prop:local} below.

Let $\gamma$ be an oriented knot in $Y$, and let $N$ be a tubular neighborhood of $\gamma$ with an identification $N\simeq S^1\times D^2$. By a ``braid in $N$ with $d$ strands'', we mean an oriented knot in $N$ which is positively transverse to the $D^2$ fibers and which intersects each fiber $d$ times.

\begin{definition}
\label{def:wb}
\begin{itemize}
\item
A {\em weighted braid} in $N$ with $m$ strands is a finite set of pairs $\zeta=\{(\zeta_i,m_i)\}$ where the $\zeta_i$ are disjoint braids in $N$ with $d_i$ strands, the $m_i$ are positive integers, and $\sum_i m_i d_i = m$.
\item
If $\tau$ is an oriented trivialization of the normal bundle of $\gamma$, then for $i\neq j$ there is a well-defined linking number $\ell_\tau(\zeta_i,\zeta_j)\in\Z$, as discussed in \S\ref{sec:topological}. Similarly, for each $i$ there is a well-defined writhe $w_\tau(\gamma_i)\in\Z$; see \cite[\S2.6]{ir}. Define the {\em writhe} of the weighted braid $\zeta$ by
\begin{equation}
\label{eqn:wwb}
w_\tau(\zeta) = \sum_im_i^2w_\tau(\zeta_i) + \sum_{i\neq j}m_i m_j\ell_\tau(\zeta_i,\zeta_j).
\end{equation}
\end{itemize}
\end{definition}

Suppose now that $\gamma$ is a simple Reeb orbit for $\lambda$, and that the normal bundle identification $N\simeq S^1\times D^2$ above is chosen so that the Reeb vector field for $\lambda$ is transverse to the $D^2$ fibers. If $\lambda'=f\lambda$ with $f$ sufficiently $C^2$ close to $1$, then the Reeb vector field for $\lambda'$ in $N$ is also transverse to the $D^2$ fibers. Suppose that this is the case.

Let $\gamma'=\{(\gamma'_k,m_k)\}$ be an orbit set for $\lambda'$ which is contained in $N$. We can regard $\gamma'$ as a weighted braid with $m$ strands for some positive integer $m$. Also note that a trivialization $\tau$ of $\gamma^*\xi$ extends to a trivialization of $\xi$ over the entire tubular neighborhood $N$, and thus canonically induces a homotopy class of trivialization $\tau'$ of $\xi$ over the Reeb orbits $\gamma_k'$. We can now state:

\begin{proposition}
\label{prop:local}
With the notation as above, if $\lambda'$ is sufficiently $C^2$ close to $\lambda$, and if $\gamma'$ is sufficiently close to $m\gamma$ as a current, then
\begin{equation}
\label{eqn:local}
\left|-w_\tau(\gamma') - \op{CZ}_{\tau'}^I(\gamma') + \sum_{l=1}^m\op{CZ}_\tau(\gamma^l)\right| \le 2m.
\end{equation}
\end{proposition}

\begin{proof}[Proof of Proposition~\ref{prop:perturb}, assuming Proposition~\ref{prop:local}]
By shrinking the tubular neighborhoods, we can assume without loss of generality that the chosen tubular neighborhood of each orbit $\alpha_i$ or $\beta_j$ has an identification with $S^1\times D^2$ with the Reeb flow of $\lambda$ transverse to the $D^2$ fibers.

In the orbit set $\alpha'$, each pair $(\alpha_i,m_i)$ gets replaced by an orbit set $\alpha_i'$ which represents a weighted braid with $m_i$ strands in the tubular neighborhood of $\alpha_i$. Likewise, each pair $(\beta_j,n_j)$ gets replaced by an orbit set $\beta_j'$ which represents a weighted braid with $n_j$ strands in the tubular neighborhood of $\beta_j$. Let $\tau$ be a homotopy class of symplectic trivializations of $\xi$ over the Reeb orbits $\alpha_i$ and $\beta_j$. As in Proposition~\ref{prop:local}, this canonically induces a homotopy class of symplectic trivializations $\tau'$ over the Reeb orbits in the orbit sets $\alpha_i'$ and $\beta_j'$.

Because $\tau$ and $\tau'$ extend to a trivialization of $\xi$ over the tubular neighborhoods containing $W_\alpha$ and $W_\beta$, it follows from the definition of the relative first Chern class that
\begin{equation}
\label{eqn:ctau}
c_{\tau'}(\alpha',\beta',Z') = c_{\tau}(\alpha,\beta,Z).
\end{equation}
By Proposition~\ref{prop:local}, if $n$ is sufficiently large then
\begin{equation}
\label{eqn:bylocal}
\begin{split}
\left| -w_\tau(\alpha_i') - \op{CZ}_{\tau'}^I(\alpha_i') + \sum_{k=1}^{m_i}\op{CZ}_\tau(\alpha_i^k)\right| \le 2m_i,\\
\left| -w_\tau(\beta_j') - \op{CZ}_{\tau'}^I(\beta_j') + \sum_{l=1}^{n_j}\op{CZ}_\tau(\beta_j^l) \right| \le 2n_j.
\end{split}
\end{equation}
By equations \eqref{eqn:I}, \eqref{eqn:ctau}, and \eqref{eqn:bylocal}, to complete the proof of Proposition~\ref{prop:perturb} it is enough to show that
\begin{equation}
\label{eqn:Qtau}
Q_{\tau'}(\alpha',\beta',Z') = Q_\tau(\alpha,\beta,Z) + \sum_iw_\tau(\alpha'_i) - \sum_jw_\tau(\beta'_j).
\end{equation}

To prove \eqref{eqn:Qtau}, by Remark~\ref{rem:additive} it is enough to show that
\[
\begin{split}
Q_\tau(\alpha',\alpha,W_\alpha) &= \sum_iw_\tau(\alpha_i'),\\
Q_\tau(\beta',\beta,W_\beta) &= \sum_jw_\tau(\beta_j').
\end{split}
\]
Since the chosen tubular neighborhoods of the Reeb orbits of $\alpha_i$ are disjoint, and the chosen tubular neighborhoods of the Reeb orbits of $\beta_j$ are disjoint, the above equations follow from Lemma~\ref{lem:Qlocal} below.
\end{proof}

\begin{lemma}
\label{lem:Qlocal}
Let $\zeta=\{(\zeta_i,m_i)\}$ be a weighted braid with $m$ strands as in Definition~\ref{def:wb}. Let $W$ be the unique relative homology class in $H_2(N,\zeta,{(\gamma,m)})$. Then
\begin{equation}
\label{eqn:Qlocal}
Q_\tau(\zeta,{(\gamma,m)},W) = w_\tau(\zeta).
\end{equation}
\end{lemma}

\noindent
Here $\tau$ defines a trivialization of the vertical tangent bundle of $N\to \gamma$ which then induces a trivialization of the normal bundle of each braid $\zeta_i$.

\begin{proof}[Proof of Lemma~\ref{lem:Qlocal}.]
We can make an admissible representative $S$ for $W$, see Definition~\ref{def:admrep}, whose intersection with $\{1-\epsilon\}\times N$ consists of $m_i$ parallel (with respect to $\tau$) copies of each $\zeta_i$, and which shrinks radially towards $\gamma$ as the $[-1,1]$ coordinate on $[-1,1]\times N$ goes down to $-1$. We can make another such admissible representative $S'$, disjoint from $S$, whose intersection with $\{1-\epsilon\}\times N$ is parallel to the first and which likewise shrinks radially towards $\gamma$. Then in equation \eqref{eqn:Q}, the intersection number term vanishes. The first linking number term in \eqref{eqn:Q} also vanishes, as it is a sum of linking numbers of braids in neighborhoods of the $\zeta_i$; for each $i$, the braid from $S$ and the braid from $S'$, with respect to $\tau$, are trivial and parallel, and thus have linking number zero. The second linking number term in \eqref{eqn:Q} is a linking number in a neighborhood of $\gamma$ and equals $w_\tau(\zeta)$. 
\end{proof}

\subsection{The structure of the braids}

To prove Proposition~\ref{prop:local}, let $\gamma$ be a simple Reeb orbit of $\lambda$, let $N$ be a tubular neighborhood of $\gamma$ as in Definition~\ref{def:wb}, and let $\tau$ be a trivialization of $\gamma^*\xi$. Let $\theta$ denote the rotation number $\theta_\tau(\gamma)\in\R$.

Suppose first that $\theta$ is irrational. Then the Reeb orbit $\gamma$ and all of its covers are nondegenerate. Consequently, when $\lambda'$ is sufficiently $C^2$ close to $\lambda$, there is a unique Reeb orbit $\gamma_0'$ for $\lambda'$ close (as a current) to $\gamma$, and for $n$ large the only possibility for the orbit set $\gamma'$ is that it is the singleton set $\gamma'=\{(\gamma_0',m)\}$. In this case Proposition~\ref{prop:local} holds because $w_\tau(\gamma')=0$ and the left hand side of \eqref{eqn:local} is zero.

The nontrivial case of Proposition~\ref{prop:local} is when the rotation number $\theta$ is rational. In this case we need to investigate the braids that can arise in $\gamma'$.  The idea in what follows is to first analyze the case where the rotation number is an integer, and then reduce the general case to this one by taking an appropriate cover of a neighborhood of $\gamma$.  

We start with the case where the rotation number is an integer.  Here there is a simple picture: each braid has just one strand, and the linking number of any two braids is given by the rotation number. More precisely:

\begin{lemma}
\label{lem:integerbraids}
With the above notation, suppose that the rotation number $\theta=a\in\Z$. Let $\lambda_n=f_n\lambda$ where $f_n\to 1$ in $C^2$. Then:
\begin{description}
\item{(a)}
For a fixed positive integer $d$, if $\{\alpha_n\}$ is a sequence where each $\alpha_n$ is a simple Reeb orbit for $\lambda_n$ in $N$ which is a braid with $d$ strands, with $\alpha_n$ converging as currents to $d\gamma$ as $n\to\infty$, then $d=1$, and in particular the writhe $w_\tau(\alpha_n)=0$ for $n$ large enough.
\item{(b)}
Given two sequences of simple Reeb orbits $\{\alpha_n\}$ and $\{\beta_n\}$ as in (a) with $\alpha_n\neq \beta_n$ for each $n$, if $n$ is sufficiently large, then the linking number $\ell_\tau(\alpha_n,\beta_n)=a$.
\end{description}
\end{lemma}

The above lemma is proved in \S\ref{sec:bangert} below. We now consider the case where the rotation number is a rational number $a/b$ that is not an integer.  Here there is a similarly nice picture: each new simple Reeb orbit that can appear can be treated for our purposes like an $(a,b)$ torus braid; see also Remark~\ref{rmk:torus}.  More precisely: 

\begin{lemma}
\label{lem:rationalbraids}
With the above notation, suppose that the rotation number is $\theta=a/b$ where $a,b$ are relatively prime integers with $b>1$. Let $\lambda_n=f_n\lambda$ where $f_n\to 1$ in $C^2$. Then:
\begin{description}
\item{(a)}
For $n$ sufficiently large, there is a unique simple Reeb orbit $\gamma_0'$ for $\lambda_n$ that is close to $\gamma$ as a current.
\item{(b)}
For a fixed integer $d>1$, if $\{\alpha_n\}$ is a sequence where each $\alpha_n$ is a simple Reeb orbit for $\lambda_n$ in $N$ which is a braid with $d$ strands, with $\alpha_n$ converging as currents to $d\gamma$ as $n\to\infty$, then $d=b$, the writhe $w_\tau(\alpha_n)=a(b-1)$, and the linking number $\ell_\tau(\gamma_0',\alpha_n)=a$.
\item{(c)}
Given two sequences of Reeb orbits $\{\alpha_n\}$ and $\{\beta_n\}$ as in (b) with $\alpha_n\neq \beta_n$ for each $n$, if $n$ is sufficiently large, then the linking number $\ell_\tau(\alpha_n,\beta_n)=ab$.
\end{description}
\end{lemma}

\begin{remark}
\label{rmk:torus}
In Lemma~\ref{lem:rationalbraids}(b), we expect that one can further show that if $n$ is sufficiently large then $\alpha_n$ is an $(a,b)$ torus braid around $\gamma_0'$; however we do not need this.
\end{remark}

\begin{proof}[Proof of Lemma~\ref{lem:rationalbraids}, assuming Lemma~\ref{lem:integerbraids}]
Part (a) holds because the Reeb orbit $\gamma$ is nondegenerate.

To prove part (b), we first note that by the same argument as for (a), we must have that $d\ge b$, because for $0<d<b$ the $d^{th}$ iterate of $\gamma$ has rotation number $da/b\notin\Z$.

Now let $\widetilde{N}$ denote the $b$-fold cyclic cover of the tubular neighborhood $N$, with the pullback of the contact form $\lambda_n$.
There is a unique simple Reeb orbit $\widetilde{\gamma_0'}$ in $\widetilde{N}$ whose projection to $N$ is a $b$-fold cover of $\gamma_0'$. In addition, by lifting the Reeb orbit $\alpha_n$ to a Reeb trajectory in $\widetilde{N}$ and extending it by the Reeb flow if needed, we obtain a simple Reeb orbit $\widetilde{\alpha_n}$ in $\widetilde{N}$ whose projection to $N$ is a cover of $\alpha_n$. By Lemma~\ref{lem:integerbraids}(a), if $n$ is sufficiently large, then $\widetilde{\alpha_n}$ is a braid with one strand in $\widetilde{N}$, hence $\alpha_n$ has at most $b$ strands. Thus $d = b$. By Lemma~\ref{lem:integerbraids}(b) we have $\ell_\tau(\widetilde{\gamma_0'},\widetilde{\alpha_n})=a$ in $\widetilde{N}$, and it follows that $\ell_\tau(\gamma_0',\alpha_n)=a$.

We now compute the writhe $w_\tau(\alpha_n)$. There are $b$ possibilities for the Reeb orbit $\widetilde{\alpha_n}$ in the previous paragraph, which we denote by $\eta_l$ for $l\in\Z/b$, ordered so that the $\Z/b$ action on $\widetilde{N}$ by deck transformations sends $\eta_l$ to $\eta_{l+1}$. The writhe $w_\tau(\alpha_n)$ is a signed count of crossings of two strands of $\alpha_n$. Each such crossing corresponds to a crossing of some $\eta_l$ with some $\eta_{l'}$ for $l\neq l'$, as well as crossings of $\eta_{l+p}$ with $\eta_{l'+p}$ for $p=1,\ldots,b-1$ obtained from the first crossing by deck transformations. On the other hand, the linking number of $\eta_l$ with $\eta_{l'}$ is one half the signed count of crossings of $\eta_l$ with $\eta_{l'}$. Thus we obtain
\[
\begin{split}
w_\tau(\alpha_n) &= \frac{1}{b}\sum_{l\neq l'}\ell_\tau(\eta_l,\eta_{l'})\\
&= \frac{1}{b}\cdot b(b-1) \cdot a\\
&= a(b-1).
\end{split}
\]
Here we are using Lemma~\ref{lem:integerbraids}(b) to get that $\ell_\tau(\eta_l,\eta_{l'})=a$ when $l\neq l'$.

We now prove (c). Similarly to the previous calculation, each crossing counted by the linking number $\ell_\tau(\alpha_n,\beta_n)$ corresponds to $b$ crossings of some lift of $\alpha_n$ (extended to a simple Reeb obit) with some lift of $\beta_n$ (extended to a simple Reeb orbit). Thus the linking number we want is $1/b$ times the sum of the linking number of each of the $b$ extended lifts of $\alpha_n$ with each of the $b$ extended lifts of $\beta_n$, which is $(1/b)\cdot b^2 \cdot a = ab$.
\end{proof}

\begin{proof}[Proof of Proposition~\ref{prop:local}]
As explained above, we can assume that $\theta=a/b$ where $a,b$ are relatively prime integers with $b>0$. When $a/b\notin\Z$, the orbit set $\gamma'$ consists of the orbit $\gamma_0'$ from Lemma~\ref{lem:rationalbraids}(a) with multiplicity $m_0$ for some $m_0\ge 0$, together with orbits $\gamma_k'$ for $k\neq 0$ with multiplicities $m_k>0$. When $a/b\in\Z$, the same is true except that we do not necessarily have a unique $\gamma_0'$ and we can take $m_0=0$. Since each $\gamma'_k$ for $k\neq 0$ is close to a $b$-fold cover of $\gamma$, we have
\begin{equation}
\label{eqn:m}
m_0+b\sum_{k\neq 0}m_k=m.
\end{equation}
By equation~\eqref{eqn:wwb} and Lemmas~\ref{lem:integerbraids} and \ref{lem:rationalbraids}, if $\lambda'$ is sufficiently $C^2$ close to $\lambda$ and if $\gamma'$ is sufficiently close to $m\gamma$ as a current, then we have
\begin{equation}
\label{eqn:comp1}
w_\tau(\gamma') = a(b-1)\sum_{k\neq 0}m_k^2 + 2am_0\sum_{k\neq 0}m_k + ab\sum_{0\neq k\neq k'\neq 0}m_km_k'.
\end{equation}

Now we consider Conley-Zehnder indices. By equation \eqref{eqn:CZI} we have
\begin{equation}
\label{eqn:useCZI}
\op{CZ}_{\tau'}^I(\gamma') = \sum_{l=1}^{m_0}\op{CZ}_{\tau'}\left((\gamma_0')^l\right) + \sum_{k\neq 0}\sum_{l=1}^{m_k}\op{CZ}_{\tau'}\left((\gamma'_k)^l\right).
\end{equation}
For a positive integer $l\le m$, if $\lambda'$ is sufficiently close to $\lambda$, then with respect to $\tau$, the Reeb orbit $(\gamma_0')^l$ has rotation number close to $(a/b)l$, and each Reeb orbit $(\gamma_k')^l$ for $k\neq 0$ has rotation number close to $al$. Then by \eqref{eqn:CZdef} and \eqref{eqn:useCZI} we get
\[
\left| \op{CZ}_{\tau'}^I(\gamma') - \sum_{l=1}^{m_0}\frac{2al}{b} - \sum_{k\neq0}\sum_{l=1}^{m_k}2al\right| \le m_0+\sum_{k\neq 0}m_k.
\]
It follows from this and \eqref{eqn:m} that
\begin{equation}
\label{eqn:comp2}
\left|\op{CZ}_{\tau'}^I(\gamma') - \frac{a}{b}(m_0^2+m_0) - a\sum_{k\neq 0}(m_k^2+m_k)\right|\le m.
\end{equation}
Finally, by \eqref{eqn:CZdef} we have
\[
\left|\sum_{l=1}^m\op{CZ}_\tau(\gamma^l) - \frac{a}{b}(m^2 + m)\right| \le m.
\]
Then by \eqref{eqn:m} we get
\begin{equation}
\label{eqn:comp3}
\left|\sum_{l=1}^m\op{CZ}_\tau(\gamma^l) - \frac{a}{b}(m_0^2+m_0) - a(2m_0+1)\sum_{k\neq 0}m_k - ab\left(\sum_{k\neq 0}m_k\right)^2\right|\le m.
\end{equation}
Combining \eqref{eqn:comp1}, \eqref{eqn:comp2}, and \eqref{eqn:comp3} gives the desired estimate \eqref{eqn:local}.
\end{proof}

\subsection{Perturbations of degenerate flows}
\label{sec:bangert}

To conclude the proof of Proposition~\ref{prop:local}, we now prove Lemma~\ref{lem:integerbraids}.

As in the statement of the lemma, let $\gamma$ be a simple Reeb orbit of $\lambda$ of period $T$, and let $\lambda_n = f_n \lambda$, where $f_n \to 1$ in $C^2$.  Let $\phi^t$ and $\phi_n^t$ denote the time $t$ flows of the Reeb vector fields for $\lambda$ and $\lambda_n$, respectively.    Let $p \in \gamma$, and let $P_{\gamma}: \xi_p \to \xi_p$ denote the linearized return map \eqref{eqn:Pgamma}.

\begin{lemma}
\label{lemma_Bangert}
Let $\{(p_n,T_n)\}_{n=1,\ldots}$ be a sequence in $Y \times (0,\infty)$ satisfying:
\begin{itemize}
\item[(c1)] $\phi_n^{T_n}(p_n) = p_n \to p$.
\item[(c2)] $\phi_n^{T_n/j}(p_n)\neq p_n$ for all integers $j\geq 2$ and all $n$.
\item[(c3)] $T_n \to T_\infty \in [0,\infty)$.
\end{itemize}
Then one of the following alternatives holds:
\begin{itemize}
\item[(a1)] $T_\infty = T$.
\item[(a2)] $T_\infty = Td$ for some integer $d\geq2$, and the eigenvalues of $P_{\gamma}$ that are roots of unity of degree $d$ generate multiplicatively all roots of unity of order $d$.
\end{itemize}
\end{lemma}

\begin{proof}
This is a special case of a result of Bangert \cite[Prop.\ 1]{bangert} for $C^1$ flows.
\end{proof}

In the situation of Lemma~\ref{lem:integerbraids}, more can be said:

\begin{corollary}
\label{cor:Bangert}
Suppose that the eigenvalues of $P_\gamma$ are real and positive. Let $\{(p_n,T_n)\}$ be a sequence satisfying conditions (c1), (c2), and (c3) of Lemma~\ref{lemma_Bangert}. Then alternative (a2) does not hold.
\end{corollary}

\begin{proof}
The only root of unity that can be an eigenvalue of $P_\gamma$ is $1$, hence the set of eigenvalues of $P_{\gamma}$ does not generate multiplicatively the group of roots of unity of order $d$ when $d\geq 2$.
\end{proof}

\begin{proof}[Proof of Lemma~\ref{lem:integerbraids}]
Part (a) follows from Corollary~\ref{cor:Bangert}.

To prove part (b), fix a diffeomorphism $\Phi$ from the tubular neighborhood $N$ of $\gamma$ to $(\R/T\Z)\times \C$ such that $\gamma$ corresponds to $(\R/T\Z)\times \{0\}$, the Reeb vector field $R_n$ of $\lambda_n$ is transverse to the $\C$ fibers for $n$ sufficiently large (assume that $n$ is this large below), and the derivative of $\Phi$ in the normal direction along $\gamma$ agrees with the trivialization $\tau$. We omit the diffeomorphism $\Phi$ from the notation below and denote points in $N$ using the coordinates $(t,z)\in(\R/T\Z)\times\C$.

By part (a), by taking $n$ large enough we can assume that $\alpha_n$ and $\beta_n$ have the same period as $\gamma$.  After reparametrization, the Reeb orbit $\alpha_n$ is given by a map
\[
\begin{split}
\R/T\Z &\longrightarrow (\R/T\Z)\times\C,\\
t & \longmapsto (t,\hat{\alpha}_n(t))
\end{split}
\]
where $\hat{\alpha}_n:\R/T\Z\to\C$. Likewise the Reeb orbit $\beta_n$ is described by a map $\hat{\beta}_n:\R/T\Z\to\C$. We have
\begin{equation}
\label{eqn:wind}
\ell_\tau(\alpha_n,\beta_n) = \op{wind}(\hat{\alpha}_n-\hat{\beta}_n),
\end{equation}
where the right hand side denotes the winding number of the loop
\[
\hat{\alpha}_n-\hat{\beta}_n:\R/T\Z\longrightarrow\C^*.
\]

We now compute the right hand side of \eqref{eqn:wind}. There is a convex neighborhood $U$ of $0$ in $\C$ such that if $n$ is sufficiently large (which we assume below), then the following two conditions hold: First, $\hat{\alpha}_n(0),\hat{\beta}_n(0)\in U$. Second, for each $t\in[0,T]$ there is a well-defined map $\psi_n^t:U\to\C$ such that for $z\in U$, the flow of the Reeb vector field $R_n$ starting at $(0,z)$ first hits $\{t\}\times\C$ at the point $(t,\psi_n^t(z))$. In particular, it follows from the definition that
\begin{equation}
\label{eqn:hats}
\begin{split}
\hat{\alpha}_n(t) &= \psi_n^t(\hat{\alpha}_n(0)),\\
\hat{\beta}_n(t) &= \psi_n^t(\hat{\beta}_n(0)).
\end{split}
\end{equation}

Now consider the derivative of $\psi_n^t$, which we denote by
\[
D\psi_n^t:U\times\C\longrightarrow \C.
\]
By \eqref{eqn:hats}, we may apply the fundamental theorem of calculus to the function 
\[
s \to \psi_n^t\left(s \hat{\alpha}_n(0) + (1-s) \hat{\beta}_n(0)\right)
\]
 to obtain
\begin{equation}
\label{eqn:ftc}
\hat{\alpha}_n(t) - \hat{\beta}_n(t) = \int_0^1 D\psi_n^t\left(s\hat{\alpha}_n(0) + (1-s)\hat{\beta}_n(0), \hat{\alpha}_n(0)-\hat{\beta}_n(0)\right) ds.
\end{equation}
By the convergence of $\lambda_n$, if $n$ is sufficiently large (which we assume below), then the amount that $D\psi^t_n(s \hat{\alpha}_n(0) + (1-s) \hat{\beta}_n(0), \cdot)$ rotates any vector as compared to $D \psi_n^t(0, \cdot)$ can be made arbitrarily small.  It follows that the integrand in \eqref{eqn:ftc}, and hence $\hat{\alpha}_n(t) - \hat{\beta}_n(t)$, has positive inner product with $D \psi_n^t(0, \hat{\alpha}_n(0) - \hat{\beta}_n(0)).$ Thus, the right hand side of \eqref{eqn:wind} differs by less than $1/4$ from the rotation number (the change in argument divided by $2\pi$) of the path
\begin{equation}
\label{eqn:path}
\begin{split}
[0,T] &\longrightarrow \C^*,\\
t &\longmapsto D\psi_n^t\left(0,\hat{\alpha}_n(0)-\hat{\beta}_n(0)\right).
\end{split}
\end{equation}
The rotation number of the linearized Reeb flow along $\gamma$ differs from the rotation number of any individual vector by less than $1/2$. Hence, by again applying convergence of the $\lambda_n$ as above, if $n$ is sufficiently large then the rotation number of the path \eqref{eqn:path} differs by less than $1/2$ from $a$.  Since the right hand side of \eqref{eqn:wind} is an integer which differs by less than $3/4$ from $a$, it must equal $a$.
\end{proof}


\section{Two simple Reeb orbits implies nondegenerate}

We now prove Theorem~\ref{thm:nondeg}. Throughout this section, assume that $Y$ is a closed connected three-manifold, and $\lambda$ is a contact form on $Y$ with exactly two simple Reeb orbits, $\gamma_1$ and $\gamma_2$, of periods $T_1$ and $T_2$ respectively.

\subsection{The homology classes of the Reeb orbits}
\label{sec:hom}

\begin{lemma}
\label{lem:torsion}
The classes $[\gamma_i] \in H_1(Y)$ and $c_1(\xi)\in H^2(Y;\Z)$ are torsion.
\end{lemma}

\begin{proof}
We use a similar argument to the proof of \cite[Thm. 1.7]{twoinf}.

Since every oriented three-manifold is spin, we can choose $\Gamma\in H_1(Y)$ such that $c_1(\xi)+ 2\op{PD}(\Gamma)=0\in H^2(Y;\Z)$. By Proposition~\ref{prop:Useq}, there exists a $U$-sequence $\{\sigma_k\}_{\ge 1}$ for $\Gamma$. Write $c_k = c_{\sigma_k}(Y,\lambda)\in\R$.

By Proposition~\ref{prop:spectral}(a), we have
\[
c_k = m_{1,k}T_1 + m_{2,k}T_2
\]
for some nonnegative integers $m_{1,k}$ and $m_{2,k}$, and furthermore
\begin{equation}
\label{eqn:repg}
m_{1,k}[\gamma_1] + m_{2,k}[\gamma_2] = \Gamma \in H_1(Y).
\end{equation}

By Proposition~\ref{prop:spectral}(b), the sequence $\{c_k\}$ is strictly increasing. It then follows from \eqref{eqn:repg} that there are infinitely many integral linear combinations of $[\gamma_1]$ and $[\gamma_2]$ that have the same value in $H_1(Y)$. Thus the kernel of the map
\begin{equation}
\label{eqn:mhom}
\begin{split}
\Z^2 &\longrightarrow H_1(Y),\\
(m_1,m_2) &\longmapsto m_1[\gamma_1]+m_2[\gamma_2]
\end{split}
\end{equation}
has rank at least $1$.

In fact, the kernel of the map \eqref{eqn:mhom} must have rank at least $2$; otherwise $c_k$ would grow at least linearly in $k$, contradicting the sublinear growth in the Volume Property in Proposition~\ref{prop:spectral}(c). It follows that $[\gamma_1]$ and $[\gamma_2]$ are torsion. Since $c_1(\xi)+2\op{PD}(\Gamma)=0$, we deduce that $c_1(\xi)$ is also torsion.
\end{proof}

\subsection{Computing the ECH index}
\label{sec:cei}

If $m_1,m_2$ are nonnegative integers, we use the notation $\gamma_1^{m_1}\gamma_2^{m_2}$ to indicate the orbit set $\{(\gamma_1,m_1),(\gamma_2,m_2)\}$, with the element $(\gamma_i,m_i)$ omitted when $m_i=0$. Write $\alpha=\gamma_1^{m_1}\gamma_2^{m_2}$. If $[\alpha] =0$, then it follows from Remark~\ref{rem:absolute} and Lemma~\ref{lem:torsion} that $I(\alpha)\in\Z$ is defined. We now give an explicit computation of $I(\alpha)$, following \cite[\S4.7]{wh}.

\begin{definition}
\label{def:rln}
Define the {\em linking number\/}
\begin{equation}
\label{eqn:rln}
\ell(\gamma_1,\gamma_2) \eqdef \frac{\ell(\gamma_1^{l_1},\gamma_2^{l_2})}{l_1l_2} \in \Q
\end{equation}
where $l_1$ and $l_2$ are positive integers such that $l_i[\gamma_i]=0\in H_1(Y)$, and on the right hand side $\ell$ denotes the usual integer-valued linking number of disjoint nullhomologous loops.
\end{definition}

\begin{definition}
\label{def:srn}
For $i=1,2$, define the {\em Seifert rotation number\/} $\phi_i\in\R$ as follows. Let $\tau$ be a trivialization of $\xi$ over $\gamma_i$. Let $\theta_{i,\tau}=\theta_\tau(\gamma_i)\in\R$ denote the  rotation number of $\gamma_i$ with respect to $\tau$. Let $l_i$ be a positive integer such that $l_i[\gamma_i]=0$. Define
\begin{equation}
\label{eqn:defqitau}
Q_{i,\tau} \eqdef \frac{Q_\tau(\gamma_i^{l_i})}{l_i^2} \in\Q,
\end{equation}
where $Q_\tau(\gamma_i^{l_i})$ is shorthand for $Q_\tau(\gamma_i^{l_i},\emptyset,Z)$ for any $Z\in H_2(Y,\gamma_i^{l_i},\emptyset)$. Note that $Q_{i,\tau}$ does not depend on $Z$ by \eqref{eqn:Qamb}, and it does not depend on $l_i$ either because $Q_\tau$ is quadratic in the relative homology class. Finally, define
\begin{equation}
\label{eqn:phii}
\phi_i \eqdef Q_{i,\tau} + \theta_{i,\tau} \in \R.
\end{equation}
The number $\phi_i$ does not depend on the choice of trivialization $\tau$, by the change of trivialization formulas in \cite[\S2]{ir}.
\end{definition}

\begin{remark}
\label{rmk:seifertframing}
When $\gamma_i$ is nullhomologous, one can alternately describe $\phi_i$ as follows.  Let $\Sigma$ be a Seifert surface spanned by $\gamma_i$. There is a distinguished homotopy class of trivialization $\tau'$ of $\xi$ over $\gamma_i$, the ``Seifert framing'', for which the normal vector to $\Sigma$ has winding number zero around $\gamma_i$. We have $Q_{i,\tau'}=0$ by \cite[Lem.\ 3.10]{ir}.  It then follows that $\phi_i =  \theta_{\tau'}(\gamma_i)$,  In the general case when $\gamma_i$ is rationally nullhomologous, one can similarly describe $\phi_i$ as the rotation number with respect to a rational framing of $\gamma_i$ determined by a rational Seifert surface.
\end{remark}

\begin{lemma}
\label{lem:cei}
If $m_1[\gamma_1]+m_2[\gamma_2]=0 \in H_1(Y)$, then
\begin{equation}
\label{eqn:Iapprox}
I(\gamma_1^{m_1}\gamma_2^{m_2}) = \phi_1m_1^2 + \phi_2m_2^2 + 2\ell(\gamma_1,\gamma_2) m_1m_2 + O(m_1+m_2).
\end{equation}
\end{lemma}

\begin{proof}
Let $l_i$ be a positive integer with $l_i[\gamma_i]=0$. Similarly to \eqref{eqn:defqitau}, define
\[
c_{i,\tau} \eqdef \frac{c_\tau(\gamma_i^{l_i})}{l_i} \in\Q,
\]
where $c_\tau(\gamma_i^{l_i})$ is shorthand for $c_\tau(\gamma_i^{l_i},\emptyset,Z)$
for any $Z\in H_2(Y,\gamma_i^{l_i},\emptyset)$. Then $c_{i,\tau}$ does not depend on $Z$ by \eqref{eqn:camb} since $c_1(\xi)$ is torsion, and it is independent of the choice of $l_i$ because $c_\tau$ is linear in the relative homology class $Z$.

It follows from the definition of the ECH index and the facts that $c_\tau$ and $Q_\tau$ are linear and quadratic in the relative homology class (see \cite[\S4.2]{wh}) that
\[
I(\gamma_1^{m_1}\gamma_2^{m_2}) = \sum_{i=1}^2(m_ic_{i,\tau}+m_i^2Q_{i,\tau}) + 2m_1m_2\ell(\gamma_1,\gamma_2) + \sum_{i=1}^2\sum_{k=1}^{m_i}(\floor{k\theta_{i,\tau}}+\ceil{k\theta_{i,\tau}}),
\]
where $Q_{i,\tau}$ and $\theta_{i,\tau}$ are as in \eqref{eqn:phii}. Plugging in the approximation
\[
\sum_{i=1}^2\sum_{k=1}^{m_i}(\floor{k\theta_{i,\tau}}+\ceil{k\theta_{i,\tau}}) = \sum_{i=1}^2m_i^2\theta_{i,\tau} + O(m_1+m_2)
\]
then gives \eqref{eqn:Iapprox}.
\end{proof}

\subsection{Using the Volume Property}

\begin{lemma}
\label{lem:quadratic}
We have
\[
\phi_i = \frac{T_i^2}{\op{vol}(Y,\lambda)},\quad\quad\quad
\ell(\gamma_1,\gamma_2) = \frac{T_1T_2}{\op{vol}(Y,\lambda)}.
\]
\end{lemma}

\begin{proof}
Both sides of the above equations are invariant under scaling the contact form by a positive constant, so we may assume without loss of generality that $\op{vol}(Y,\lambda)=1$.

By Proposition~\ref{prop:Useq} and Lemma~\ref{lem:torsion}, we can choose a $U$-sequence $\{\sigma_k\}_{k\ge 1}$ for $\Gamma=0$. Since the $U$ map has degree $-2$, there is a constant $C\in\Z$ such that for each positive integer $k$ the class $\sigma_k$ has grading $C+2k$. By Proposition~\ref{prop:spectral}(a), for each positive integer $k$ there are nonnegative integers $m_{1,k}, m_{2,k}$ such that
\begin{equation}
\label{eqn:csk}
c_{\sigma_k}(Y,\lambda) = m_{1,k}T_1 + m_{2,k}T_2.
\end{equation}
By the Volume Property of Proposition~\ref{prop:spectral}(c), we have
\begin{equation}
\label{eqn:2k1}
2k = (m_{1,k}T_1 + m_{2,k}T_2)^2 + o(k).
\end{equation}

Fix $k$ and write $\alpha_k=\gamma_1^{m_{1,k}}\gamma_2^{m_{2,k}}$. If $\lambda'$ is a sufficiently $C^2$ close nondegenerate perturbation of $\lambda$, then by the same compactness argument that proves Proposition~\ref{prop:spectral}(a), there is an orbit set $\alpha_k'$ close to $\alpha_k$ as a current such that $I(\alpha_k')=C+2k$ (and also $\int_{\alpha_k'}\lambda'$ is close to $c_{\sigma_k}(Y,\lambda)$, although we do not need this). By Proposition~\ref{prop:perturb}, we have
\[
C + 2k = I(\alpha_k) + O(m_{1,k} + m_{2,k}).
\]
Combining this with Lemma~\ref{lem:cei}, we get
\begin{equation}
\label{eqn:2k2}
2k = \phi_1m_{1,k}^2 + \phi_2m_{2,k}^2 + 2\ell(\gamma_1,\gamma_2)m_{1,k}m_{2,k} + O(m_{1,k} + m_{2,k}).
\end{equation}

Putting together \eqref{eqn:2k1} and \eqref{eqn:2k2}, we obtain
\[
(\phi_1-T_1^2)m_{1,k}^2 + (\phi_2-T_2^2)m_{2,k}^2 + 2(\ell(\gamma_1,\gamma_2)-T_1T_2)m_{1,k}m_{2,k} = O(m_{1,k}+m_{2,k}) + o(k).
\]
Consequently, if the sequence $(m_{2,k}/m_{1,k})_{k\ge 1}$ has an accumulation point $S\in[0,\infty]$, then the line in the $x,y$ plane of slope $S$ through the origin is in the null space of the quadratic form
\[
f(x,y) = (\phi_1-T_1^2)x^2 + (\phi_2-T_2^2)y^2 + 2(\ell(\gamma_1,\gamma_2)-T_1T_2)xy.
\]
To complete the proof of the lemma, it now suffices to show that the sequence $(m_{2,k}/m_{1,k})_{k\ge 1}$ has at least three accumulation points, as then the quadratic form $f$ must vanish identically. We claim that in fact this sequence has infinitely many accumulation points.

If the sequence has only finitely many accumulation points $S_1,\ldots,S_n$, then for every $\epsilon>0$, there exists $R>0$ such that every point $(m_{1,k},m_{2,k})$ is contained in the union of the disk $x^2+y^2\le R^2$ and the cones around the lines of slope $S_1,\ldots,S_n$ with angular width $\epsilon$.

Since $\lim_{k\to\infty}c_{\sigma_k}^2/k=2$, and since the points $(m_{1,k},m_{2,k})$ are pairwise distinct by Proposition~\ref{prop:spectral}(b), by equation \eqref{eqn:csk}, it follows that for large $L$, the number of points $(m_{1,k},m_{2,k})$ contained in the triangle $T_1x+T_2y\le L, \; x\ge 0, \; y\ge 0$ is approximately $L^2/2$. As a result, there exists $\delta>0$ such that for all $L$ sufficiently large, the fraction of lattice points in the above triangle that are contained in the sequence $(m_{1,k},m_{2,k})_{k\ge 1}$ is at least $\delta$. This gives a contradiction if $\epsilon$ in the previous paragraph is chosen sufficiently small.
\end{proof}

\subsection{Completing the proof of nondegeneracy}
\label{sec:cpn}

\begin{proof}[Proof of Theorem~\ref{thm:nondeg}]
The ratio $T_1/T_2$ is irrational\footnote{The proof is simple: If $T_1/T_2$ is rational, so that $T_1$ and $T_2$ are both integer multiples of a single number, then Proposition~\ref{prop:spectral}(b) implies that the spectral invariants associated to a $U$-sequence grow at least linearly, contradicting the Volume Property.} by \cite[Thm.\ 1.3]{onetwo}.  Also, $\ell(\gamma_1,\gamma_2)$ is rational by the definition \eqref{eqn:rln}. It then follows from Lemma~\ref{lem:quadratic} that $\phi_1$ and $\phi_2$ are irrational.

By \eqref{eqn:phii}, since $Q_{i,\tau}$ is rational, it follows that the rotation number $\theta_{i,\tau}$ is irrational. Then $P_{\gamma_i}$ has eigenvalues $e^{\pm 2\pi i\theta_{i,\tau}}$, so the Reeb orbits $\gamma_i$ are irrationally elliptic. As explained in \S\ref{sec:sor}, it follows that all covers of $\gamma_i$ are nondegenerate, so $\lambda$ is nondegenerate.
\end{proof}


\section{Additional dynamical information}
\label{sec:fc}

To finish up, we now prove Theorem~\ref{thm:dynamics}.

To prepare for the proof, recall that if $Y$ is a closed oriented three-manifold, if $\xi$ is a contact structure on $Y$ with $c_1(\xi)=0\in H^2(Y;\Z)$, and if $\gamma$ is a nullhomologous transverse knot, then the {\em self-linking number\/} $\op{sl}(\gamma)\in\Z$ is defined to be the difference between the Seifert framing (see Remark~\ref{rmk:seifertframing}) and the framing given by a global trivialization of $\xi$. In the notation of \S\ref{sec:cei}, we have
\begin{equation}
\label{eqn:sl}
\op{sl}(\gamma) = Q_\tau(\gamma) - c_\tau(\gamma)
\end{equation}
where $\tau$ is any trivialization of $\xi|_\gamma$.

Now suppose that $\gamma$ above is a simple Reeb orbit. Let $\phi(\gamma)\in\R$ denote the rotation number of $\gamma$ with respect to the Seifert framing as in \S\ref{sec:cei}, and let $\theta(\gamma)\in\R$ denote the rotation number of $\gamma$ with respect to a global trivialization of $\xi$. Also let
\[
\op{CZ}(\gamma) = \floor{\theta(\gamma)}+\ceil{\theta(\gamma)}\in \Z
\]
denote the Conley-Zehnder index of $\gamma$ with respect to a global trivialization.
It follows from \eqref{eqn:sl} that
\begin{equation}
\label{eqn:phitheta}
\phi(\gamma) = \theta(\gamma) + \op{sl}(\gamma).
\end{equation}

\begin{proof}[Proof of Theorem~\ref{thm:dynamics}]
By Corollary~\ref{cor:lensspace}, $\gamma_1$ and $\gamma_2$ are the core circles of a genus one Heegaard splitting of $Y$. It follows from this topological description that $\ell(\gamma_1,\gamma_2)=1/p$. Part (a) of the theorem then follows from Lemma~\ref{lem:quadratic}.

To prove part (b), suppose first that $Y = S^3$. We know from Theorem~\ref{thm:nondeg} that $\lambda$ is nondegenerate and there are no hyperbolic Reeb orbits. Then $\xi$ is tight, because otherwise \cite[Thm.\ 1.4]{unknotted} would give a hyperbolic Reeb orbit. Moreover, it follows from \cite[Thm.\ 1.3]{HS}, combined with \cite[Thm.\ 1.4]{unknotted} and the fact that there are no Reeb orbits with $\op{CZ}=2$ (since Reeb orbits with even Conley-Zehnder index have integer rotation number and thus are hyperbolic), that one of the simple Reeb orbits, say $\gamma_1$, satisfies ${\rm sl}(\gamma_1)=-1$ and $\op{CZ}(\gamma_1)=3$, and is the binding of an open book decomposition with pages that are disk-like global surfaces of section for the Reeb flow. The return map on a page preserves an area form with finite total area, hence it has a fixed point by Brouwer's Translation Theorem. This fixed point corresponds to the simple Reeb orbit $\gamma_2$, which is transverse to the pages of the open book. Since on $S^3\setminus\gamma_1$ the tangent spaces of the pages define a distribution that is isotopic to $\xi$ keeping transversality with the Reeb direction, we get ${\rm sl}(\gamma_2)=-1$. Since $\op{CZ}(\gamma_1)=3$, we have $\theta(\gamma_1)\in(1,2)$, so by equation \eqref{eqn:phitheta} we have $\phi_1\in(0,1)$. By Lemma~\ref{lem:quadratic} as used in equation \eqref{eqn:actionindex}, we have $\phi_1\phi_2=1$, so $\phi_2>1$. By equation \eqref{eqn:phitheta} again we have $\theta(\gamma_2)>2$. It follows that all iterates of $\gamma_1$ and $\gamma_2$ have $\theta > 1$, so $\lambda$ is dynamically convex.

To prove part (b) in the general case, let $\tilde\lambda$ denote the pullback of the contact form $\lambda$ to the universal cover $S^3$ of $Y$. It follows from the Heegaard decomposition that $\gamma_1$ and $\gamma_2$ each have order $p$ in $\pi_1(Y)$. Consequently $\widetilde{\lambda}$ has exactly two simple Reeb orbits $\widetilde{\gamma_1}$ and $\widetilde{\gamma_2}$, which project to $\gamma_1$ and $\gamma_2$ as $p$-fold coverings. By the previous paragraph, $(S^3,\widetilde{\lambda})$ is dynamically convex and tight, and it follows that $(Y,\lambda)$ is dynamically convex and universally tight.
\end{proof}


\bigskip

{\footnotesize

{\sc Dan Cristofaro-Gardiner}

University of California, Santa Cruz, and

School of Mathematics, Institute for Advanced Study, Princeton NJ, USA

{\em dcristof@ucsc.edu}

\medskip

{\sc Umberto Hryniewicz}

RWTH Aachen, Jakobstrasse 2, Aachen 52064, Germany

{\em hryniewicz@mathga.rwth-aachen.de}

\medskip

{\sc Michael Hutchings}

University of California, Berkeley

{\em hutching@math.berkeley.edu\/}

\medskip

{\sc Hui Liu}

School of Mathematics and Statistics, Wuhan University, Wuhan 430072, Hubei, P.R. China

{\em huiliu00031514@whu.edu.cn}

}

\end{document}